\documentclass [leqno] {amsart}
  
\usepackage {amssymb, amsthm} 

\usepackage[all]{xy}
\usepackage {hyperref} 
\usepackage{xcolor}

\theoremstyle{plain}

\newtheorem{lem}[subsection]{Lemma}
\newtheorem{cor}[subsection]{Corollary}
\newtheorem{thm}[subsection]{Theorem}

\theoremstyle{definition}

\theoremstyle{remark}
\newtheorem{rmk}[subsection]{Remark}

\numberwithin{equation}{section}

\newcommand{\eq}[2]{\begin{equation}\label{#1}#2 \end{equation}}

\newcommand{\mlnl}[1]{\begin{multline*}#1 \end{multline*}}

\newcommand{\arir}{\ar@{^{(}->}}
\newcommand{\aril}{\ar@{_{(}->}}
\newcommand{\are}{\ar@{>>}}

\newcommand{\xr}[1] {\xrightarrow{#1}}

\newcommand{\inj}{\hookrightarrow}
\newcommand{\surj}{\twoheadrightarrow}

\newcommand{\ul}{\underline}

\newcommand{\Div}{{\rm div}}

\newcommand{\Spec}{{\rm Spec \,}}

\newcommand{\Tr}{{\rm Tr}}

\newcommand{\Res}{{\rm Res}}

\newcommand{\sE}{{\mathcal E}}
\newcommand{\sF}{{\mathcal F}}

\newcommand{\sM}{{\mathcal M}}
\newcommand{\sN}{{\mathcal N}}
\newcommand{\sO}{{\mathcal O}}

\newcommand{\G}{{\mathbb G}}

\newcommand{\Z}{\mathbb{Z}}

\newcommand{\fm}{\mathfrak{m}}

\begin{document}

\title[$K$-groups of reciprocity functors]
{$K$-groups of reciprocity functors \\ 
for $\mathbb{G}_a$ and abelian varieties}

%%%%%%%%%%
%%%%%%%%%% 

\author{Kay R\"ulling \and Takao Yamazaki}
\date{\today}
\address{Freie Universit\"at Berlin, Arnimallee 3, 14195 Berlin, Germany}
\email{kay.ruelling@fu-berlin.de}
\address{Mathematical Institute, Tohoku University,
  Aoba, Sendai 980-8578, Japan}
\email{ytakao@math.tohoku.ac.jp}

\begin{abstract}
We prove that the $K$-group of  reciprocity functors,
defined by F. Ivorra and the first author,
vanishes over a perfect field as soon as one of the reciprocity functors is
$\G_a$ and one is an abelian variety.
\end{abstract}

\keywords{Reciprocity functor, Chow group, $K$-group}
\subjclass[2010]{19E15 (14F42, 19D45, 19F15)}
\thanks{ The first author is supported by the ERC Advanced Grant 226257,
the second author is supported by JSPS KAKENHI Grant (22684001, 24654001).}
\maketitle

\section{Introduction}
We work over a perfect base field $k$
of characteristic $p \geq 0$.
Stemming from B. Kahn's idea \cite{KahnRF},
F. Ivorra and the first author developed
a theory of {\em reciprocity functors} in \cite{IvRu}.
We recall relevant facts from \cite{IvRu} in \S \ref{sect:review} below.
Here we only mention a reciprocity functor
restricts to a covariant functor $\sE_k \to ({\rm Ab})$,
where $\sE_k$ is the category of
finitely generated field extensions of $k$.
Here are examples of reciprocity functors:
\begin{itemize}
\item[(i)]
A commutative algebraic group $G$ over $k$
can be regarded as a reciprocity functor.
For any $L \in \sE_k$,
$G(L)$ is the group of $L$-rational points of $G$.
\item[(ii)]
Let $X$ be a smooth projective $k$-scheme.
We then have a reciprocity functor $\ul{CH}_0(X)$
such that
\[ \ul{CH}_0(X)(L) = CH_0(X \times_k L)
\text{ for any $L \in \sE_k$}.
\]
This construction is covariant functorial in $X$.
(The assumption on $X$ can be relaxed,
see \S \ref{ex:rf} below.)
\item[(iii)]
For $X$ as in (ii), we define 
\begin{equation*}
\ul{CH}_0(X)^0 := \ker(\pi_* : \ul{CH}_0(X) \to \ul{CH}_0(\Spec k_X)),
\end{equation*}
where $k_X :=H^0(X, \sO_X)$
and $\pi : X \to \Spec k_X$ is the canonical map.
If $X$ is connected and admits a zero-cycle of degree one,
then we have a decomposition
$\ul{CH}_0(X) \cong \ul{CH}_0(X)^0 \oplus \Z$
where
$\Z := \ul{CH}_0(\Spec k)$. 
\end{itemize}
Given a family $\sM_1, \dots, \sM_n$ of reciprocity functors,
a new reciprocity functor $T(\sM_1, \dots, \sM_n)$,
called the {\em $K$-group of reciprocity functors},
is constructed in \cite{IvRu}.
The nature of this construction is 
illustrated by the following examples:
\begin{itemize}
\item
Let $L \in \sE_k$.
If we take $\sM_1=\dots=\sM_n=\G_m$, 
then $T(\G_m, \dots, \G_m)(L)$
is isomorphic to 
the Milnor $K$-group $K_n^M(L)$
\cite[5.3.3]{IvRu}.
Suppose moreover $k$ is of characteristic zero.
If we take $\sM_1=\G_a, ~\sM_2=\dots=\sM_n=\G_m$, 
then $T(\G_a, \G_m, \dots, \G_m)(L)$
is isomorphic to 
the group of K\"ahler differentials $\Omega_{L/\Z}^{n-1}$
\cite[5.4.7]{IvRu}.
\item 
If $p \not= 2$,
then we have
$T(\G_a, \G_a, \sM_1, \dots, \sM_n)=0$
for any reciprocity functors $\sM_1, \dots, \sM_n$.
\item 
Let $X_1, \dots, X_n$ be smooth projective schemes over $k$
of pure dimensions $d_1, \dots, d_n$ and let $r \in \Z_{\geq 0}$.
Then we have an isomorphism for any $L \in \sE_k$
\[ T(\ul{CH}_0(X_1), \dots, \ul{CH}_0(X_n), \G_m, \dots, \G_m)(L)
\cong CH^{d+r}(X \times_k L, r).
\]
where $X:=X_1 \times_k \dots \times_k X_n, ~d:=d_1+\dots+d_n$
and we put $r$ copies of $\G_m$ on the left hand side
\cite[5.2.5]{IvRu}.
\end{itemize}

Our main result is the following theorem.

\begin{thm}\label{thm:main}
Let $k$ be a perfect field and
let $\sM_1,\ldots, \sM_n$ be reciprocity functors.
For any smooth projective $k$-scheme $X$ 
and for any {\em perfect} $L \in \sE_k$, we have
\begin{equation}\label{eq:vanish}
T(\G_a, \ul{CH}_0(X)^0,\sM_1,\ldots, \sM_n)(L)=0.
\end{equation}
\end{thm}

It was suggested by B. Kahn that one should have $T(\G_a, A)=0$ for an abelian variety $A$.
The above theorem immediately implies the following result,
which in particular proves this conjecture in characteristic zero.

\begin{cor}\label{cor-to-thm:main}
Under the same assumption as Theorem \ref{thm:main},
for any abelian variety $A$ and for any perfect $L \in \sE_k$
we have 
\begin{equation}\label{eq:vanish-ab}
T(\G_a, A, \sM_1,\ldots, \sM_n)(L)=0.
\end{equation}
\end{cor}

\begin{rmk}
\begin{enumerate}
\item
If $p=0$,
we can apply Theorem \ref{thm:main} for any $L \in \sE_k$
and hence
$T(\G_a, A,\sM_1,\ldots, \sM_n)=0$
as a reciprocity functor.
If $p>0$,
however,
we cannot get the same conclusion,
because
we do not know if \eqref{eq:vanish} holds
when $L$ is an imperfect field.
\item
In Theorem \ref{thm:main}, 
assume additionally that $X$  has a zero-cycle of degree one, then we obtain 
\[T(\G_a, \ul{CH}_0(X),\sM_1,\ldots, \sM_n)(L)=T(\G_a, \sM_1,\ldots, \sM_n)(L).\]
This result should be compared with \cite[Thm 4.5]{H}, where Hiranouchi proves 
 a similar statement for his $K$-groups. Notice that the definition of his $K$-groups and the $K$-groups of
reciprocity functors recalled above are completely different and so are the proofs.
\item
Theorem \ref{thm:main} can be extended to
more general $X$.
See Theorem \ref{thm:general} for details.
\end{enumerate}
\end{rmk}

After a brief review of \cite{IvRu} in \S \ref{sect:review},
we prove Theorem \ref{thm:main}
when $\dim X=1$ in \S \ref{sect:curves}.
This is the main step in the proof.
The general case of Theorem \ref{thm:main} 
and Corollary \ref{cor-to-thm:main}
will be deduced from this case
by a Bertini-type theorem in \S \ref{sect:highdim}.

\subsection*{Acknowledgement}
The authors thank Bruno Kahn for his comments on a first draft of this note.
Thanks are also due to the referee for many helpful comments.

\section{Review of reciprocity functors}
\label{sect:review}

\subsection{Generalities}
Let ${\rm Reg}^{\le 1}$ be the category of
regular $k$-schemes of dimension $\leq 1$, 
which are separated and of finite type over some $L \in \sE_k$.
The category having the same objects as 
${\rm Reg}^{\le 1}$  but 
finite correspondences as morphisms
will be denoted by ${\rm Reg}^{\le 1}{\rm Cor}$.
A reciprocity functor $\sM$ is a 
contravariant functor $\sM : {\rm Reg}^{\le 1}{\rm Cor} \to ({\rm Ab})$
satisfying various conditions
which we do not recall here.
For details, we refer the readers to \cite{IvRu}.
Here we recall a few properties of reciprocity functors.

Before stating them, we introduce a few notations.
Let $C \in {\rm Reg}^{\le 1}$ be a connected
regular projective curve over $L \in \sE_k$
and $x \in C$ a closed point.
We write $v_x$ for the normalized discrete valuation 
on $L(C)$ defined by $x$
and we put $U_{C, x}^{(n)} := 
\{ f \in L(C)^{\times} ~|~ v_x(f-1) \geq n \}$
for each $n \in \Z_{>0}$.
Let $\sM$ be a reciprocity functor.
We define $\sM_{C, x} := \underset{\longrightarrow}{\lim} \sM(U)$, 
where the limit is taken over all open neighborhoods $U$ of $x$,
and write 
\begin{equation}\label{eq:specializaition}
s_x : \sM_{C, x} \to \sM(x)
\qquad
(\text{resp.}~ \Tr_{x/L} : \sM(x) \to \sM(L) := \sM(\Spec L))
\end{equation}
for the map induced 
by pull-back along the embedding $x \to U$ for each $U$
(resp. by push-forward along $x \to \Spec L$).
When $x=\Spec L'$, we also write $\Tr_{L'/L}=\Tr_{x/L}$.

\subsection{Injectivity}
Let $\sM$ be a reciprocity functor
and 
let $C \in {\rm Reg}^{\le 1}$ be connected
with generic point $\eta$.
The definition of reciprocity functors imposes
the induced map $\sM(C) \to \sM(U)$ to be injective
for any non-empty open subset $U \subset C$.
Hence we may view 
$\sM(C)$ as a subgroup of $\sM(\eta)$.

\subsection{Local symbol}
Let $\sM$ be a reciprocity functor
and let $C \in {\rm Reg}^{\le 1}$ be a
smooth projective geometrically connected curve over $L \in \sE_k$.
Then for all closed points $x \in C$ 
there exists a biadditive pairing
called the {\em local symbol}
\begin{equation}\label{eq:local-symbol}
 (-,-)_x : \sM(L(C)) \times L(C)^\times \to \sM(L)
\end{equation}
which has the following properties
\cite[1.5.3]{IvRu}:
\begin{itemize}
\item
For all $m \in \sM(L(C))$,
there exists $n \in \Z_{>0}$ such that
$(m, U_{C, x}^{(n)})_x=0$.
\item
For all $m \in \sM_{C, x}$
and $f \in L(C)^\times$, we have
$(m, f)_x = v_x(f) \Tr_{x/L}(s_x(m))$.
\item
For all $m \in \sM(L(C))$ and $f \in L(C)^\times$,
we have
\begin{equation}\label{eq:reciprocity}
 \sum_{x \in C} (m, f)_x = 0.
\end{equation}
\end{itemize}
Using the local symbol, we define
${\rm Fil}^0_x \sM(L(C)) := \sM_{C, x}$ and
for $n \in \Z_{>0}$
\begin{equation}\label{eq:def-filtrarion}
{\rm Fil}^n_x \sM(L(C)) := 
\{ m \in \sM(L(C)) ~|~ (m, f)_x = 0 ~\text{for all }
f \in U_{C, x}^{(n)} \}.
\end{equation}

\subsection{$K$-group of reciprocity functors}\label{sect:k-rf}
Let $\sM_1, \dots, \sM_n$ be reciprocity functors.
The $K$-group of reciprocity functors
$T(\sM_1, \dots, \sM_n)$
is a reciprocity functor 
constructed in \cite[4.2.4]{IvRu}.
We shall need the following properties of 
$T(\sM_1, \dots, \sM_n)$.
\begin{itemize}
\item[(i)]
Let $L \in \sE_k$. There is a multi-additive map
\[ \sM_1(L) \times \dots \times \sM_n(L) \to 
T(\sM_1, \dots, \sM_n)(L).
\]
The image of 
$(m_1, \dots, m_n) \in \sM_1(L) \times \dots \times \sM_n(L)$
is denoted by
$m_1 \otimes \dots \otimes m_n \in T(\sM_1, \dots, \sM_n)(L)$.
\item[(ii)]
$T(\sM_1, \dots, \sM_n)(L)$ is generated by
the elements of the form
\begin{equation}\label{eq:generate}
 \Tr_{L'/L}(m_1' \otimes \dots \otimes m_n') 
\end{equation}
where $L'$ ranges all finite extensions of $L$
and
$m_i' \in \sM_i(L') ~(i =1, \dots, n)$.
Take such an $L'$ and $j \in \{ 1, \dots, n \}$.
For any $m_i \in \sM_i(L) ~(i \not= j)$ and $m_j' \in \sM_j(L')$,
we have a {\it projection formula}
\begin{equation}\label{eq:proj-formula}
 \Tr_{L'/L}(\pi^* m_1 \otimes \dots \otimes m_j' 
  \otimes \dots \otimes \pi^* m_n)
 = m_1 \otimes \dots \otimes \Tr_{L'/L}(m_j') \otimes \dots \otimes m_n
\end{equation}
where 
$\pi$ denotes the natural map $\Spec L' \to \Spec L$.
\item[(iii)]
Let $C \in {\rm Reg}^{\le 1}$ be a
smooth projective geometrically connected curve over $L \in \sE_k$,
$x\in C$ a closed point and $r \in \Z_{\geq 0}$.
If $m_i \in {\rm Fil}_x^r \sM_i(L(C))$ for $i=1, \dots, n$,
then we have
\begin{equation}\label{eq:conti}
m_1 \otimes \dots \otimes m_n \in 
{\rm Fil}^r_x T(\sM_1, \dots, \sM_n)(L(C)).
\end{equation}
Further, if $r=0$, then in  $T(\sM_1, \dots, \sM_n)(k(x))$ we have
\[s_x(m_1\otimes\ldots\otimes m_n)=s_x(m_1)\otimes\ldots \otimes s_x(m_n).\]
\end{itemize}

\subsection{Examples of reciprocity functors}\label{ex:rf}
As was recalled in the introduction,
a commutative algebraic group defines a reciprocity functor
\cite[2.2.2]{IvRu}.
Also, 
a homotopy invariant Nisnevich sheaves with transfers
$\sF$ (in the sense of \cite{voepre})
defines a reciprocity functor \cite[2.3.5]{IvRu}
which we denote by the same letter $\sF$.
We shall apply this construction
to $\sF=\ul{CH}_0(X)$ appearing in the following result.
\begin{thm}[{Huber-Kahn, \cite[2.2]{HK}}]\label{thm:huber-kahn}
Let $X$ be a scheme separated and of finite type over $k$.
We assume one of the following conditions:
\begin{enumerate}
\item $X$ is smooth and projective over $k$.
\item $X$ is of dimension $\leq 2$.
\item $k$ is of characteristic zero.
\end{enumerate}
Then there is 
a homotopy invariant Nisnevich sheaf with transfers
$\ul{CH}_0(X)$ such that
for any connected smooth scheme $U$ over $k$
\[ \ul{CH}_0(X)(U)= CH_0(X \times_k k(U)).
\]
\end{thm}
Suppose $X$ is proper over $k$
and put $k_X :=H^0(X, \sO_X)$.
The kernel of 
$\ul{CH}_0(X) \to \ul{CH}_0(\Spec k_X)$
(induced by the pushforward along the canonical map $X \to \Spec k_X$)
is denoted by $\ul{CH}_0(X)^0$.

\begin{rmk}\label{rmk:sp-eta}
Let $C$ be a connected smooth projective curve over $k$
with generic point $\eta$.
We can view $\eta$ canonically as a zero-cycle on $C_{\eta}$, 
i.e. $\eta \in CH_0(C_{\eta}) = \ul{CH}_0(C)(\eta)$. 
Since the restriction map
$\ul{CH}_0(C)(U) \to \ul{CH}_0(C)(\eta)$
is an isomorphism
for any open dense subscheme $U \subset C$,
we have $\sM_{C, x}=\ul{CH}_0(C)(\eta)$
and
\begin{equation}\label{eq:sx-eta-x}
s_x(\eta)=x \quad \text{in}~ \ul{CH}_0(C)(x)=CH_0(C \times_k x)
\end{equation}
(see \eqref{eq:specializaition})
for any closed point $x \in C$.
\end{rmk}

\begin{rmk}\label{rem:reduction-to-k}
Let $L\in \sE_k$ be a perfect field. Given a reciprocity functor $\sN$ over $k$,
we denote by $\sN_L$ the reciprocity functor over $L$
obtained by restricting $\sN$.

Let $\sM_1, \dots, \sM_n$ be reciprocity functors over $k$.
Then it follows directly from the universal property of the $K$-groups
see \cite[4.2.4]{IvRu} that we have a natural surjection of reciprocity functors over $L$
\[T(\sM_{1, L}, \dots, \sM_{n, L})
\surj 
T(\sM_1, \dots, \sM_n)_L,\]
sending a symbol $m_1\otimes\ldots\otimes m_n$ viewed as an element in the left hand side
to the same symbol viewed as an element in the right hand side.
In fact it is not hard to show that this map is an isomorphism. But we do not need this in the following
and therefore leave the proof to the interested reader.
 \end{rmk}

\begin{lem}\label{lem:surj}
Let $\sM_1, \dots, \sM_n$ be reciprocity functors.
For any finite field extension $L/k$,
the trace map
\[\Tr_{L/k}: T(\G_a,\sM_1, \dots, \sM_n)(L)
\to T(\G_a,\sM_1, \dots, \sM_n)(k)
\]
is surjective.
\end{lem}
\begin{proof}
Since $k$ is perfect the trace map 
$\Tr_{L \otimes_k k'/k'}:\G_a(L \otimes_k k')\to \G_a(k')$ is surjective 
for any field extension $k'/k$
(see \cite[Chap. III, Cor. 1 to Prop. 4]{weil}).
Hence the lemma follows from 
the projection formula \eqref{eq:proj-formula} 
and \S \ref{sect:k-rf} (ii).
\end{proof}

\begin{lem}\label{rem:base-change}
Assume $X$ is proper over $k$ and 
assume one of the conditions in Theorem \ref{thm:huber-kahn}.
For any finite field extension $L/k$, 
there is a canonical isomorphism
$(\ul{CH}_0(X)^0)_L \cong \ul{CH}_0(X\times_{k}L)^0$
(as reciprocity functors over $L$).
\end{lem}
\begin{proof}
We have $\ul{CH}_0(X)_L = \ul{CH}_0(X\times_{k}L)$
by definition.
Since
$H^0(X, \sO_X) \otimes_k L \cong H^0(X \times_k L, \sO_{X \times_k L})$
(see \cite[p. 53, Cor. 5]{Mumford}),
for any $U\in {\rm Reg}^{\le 1}$ over $L$
we have 
$k_X \times_k k(U) = (k_{X \times_k L}) \times_L k(U)$.
This implies
$\ul{CH}_0(\Spec k_X)_L = \ul{CH}_0(\Spec k_{X\times_{k}L})$,
and we are done.
\end{proof}

\section{Case of curves}
\label{sect:curves}

In this section, we prove
Theorem \ref{thm:main} assuming $\dim X=1$. 
In what follows we use the abbreviation 
$T(\G_a,\ul{CH}_0(X)^0,\ul{\sM}) 
:=T(\G_a,\ul{CH}_0(X)^0,\sM_1,\ldots, \sM_n)$. 
%Notice that for a finite field extension $L/k$, we have (with the notation 
%from Remark \ref{rem:reduction-to-k})
%$\ul{CH}_0(X)^0_L= \ul{CH}_0(X\times_{k}L)^0$.
By Remark \ref{rem:reduction-to-k} and Lemmas \ref{lem:surj}, \ref{rem:base-change},
we may assume 
that $X$ has a $k$-rational point.
We may further assume that $X$ is connected.
Thus we have a splitting of reciprocity functors
$\ul{CH}_0(X)\cong \ul{CH}_0(X)^0\oplus \Z$,
which implies the injectivity of the natural map
\begin{equation}\label{eq:split}
T(\G_a,\ul{CH}_0(X)^0,\ul{\sM})
\hookrightarrow 
T(\G_a,\ul{CH}_0(X),\ul{\sM}).
\end{equation}
By \S \ref{sect:k-rf} (ii) and Remark \ref{rem:reduction-to-k},
it suffices to show that for any elements 
$a\in k\setminus\{0\}$, $m_i\in \sM_i(k)$ 
and any zero-cycle $\zeta=\sum_{i=1}^r n_i x_i\in Z_0(X)$ 
with $\sum_i n_i [k(x_i): k]=0$ we have
\[a\otimes\zeta\otimes m_1\otimes\ldots\otimes m_n=0 \quad \text{in }
T(\G_a,\ul{CH}_0(X),\ul{\sM})(k). 
\]
Note that we use $\ul{CH}_0(X)$ instead of $\ul{CH}_0(X)^0$.
This is justified by \eqref{eq:split}.

Given $N \in \Z$,
we write $(N, p)=1$ to mean 
``$N \not= 0$'' if $p=0$ and
``$N$ is prime to $p$'' if $p>0$.
Since $T(\G_a,\ul{CH}_0(X), \ul{\sM})(k)$ is annihilated by $p$,
we may assume $(n_i, p)=1$ for all $i$. Further we may assume  $n_1=1$. 
Else replace $\zeta$ by $\zeta+{\rm div}(h)$, where we choose $h\in k(X)^\times$ with $v_{x_1}(h)=n_1-1$.

By the Approximation Lemma we find a function $f\in k(X)^\times$ such that
\eq{11}{v_{x_i}(f)=n_i,\quad \text{for all } i=1,\ldots, r. }
Since $n_1=1$ the function $f$ is a local parameter at $x_1$ and since $k(x_1)/k$ is separable 
there exists an open neighborhood of $x_1$ in $X$ which is \'etale over $\Spec k[f]$.
In particular the function $f$ is a separating transcendence basis of $k(X)/k$ and thus 
$\Omega^1_{k(X)/k}=k(X)\frac{df}{f}$. 

For a closed point $x\in X$ set 
\[H^1_x:=\frac{\Omega^1_{k(X)/k}}{\Omega^1_{X,x/k}}.\] 
We have the residue map  $\partial_x: H^1_x\to k(x)$ at our disposal (cf. \cite[Chap. II]{Se}) and we denote by
$\Res_x=\Tr_{k(x)/k}\circ\partial_x : H^1_x\to k$ the composition with the trace. 
By duality theory (see e.g. \cite[III, Rmk 7.14]{Ha}) we have an exact sequence
\eq{12}{\Omega^1_{k(X)/k}\to \bigoplus_{x\in X} H^1_x\xr{\sum_x \Res_x} k \to 0. }
Let $\alpha=(\alpha_x)_{x\in X} \in \bigoplus_{x\in X} H^1_x$ be the element given by
\[\alpha_x=\begin{cases}  n_i \frac{d t_i}{t_i}, & \text{if } x=x_i, \text{ some } i\in \{1,\ldots, r\},\\
                                                 0,                     & \text{else}, \end{cases}\]
where $t_i\in \sO_{X,x_i}$ is a local parameter at $x_i$.

We have
\[{\rm Res}_{x_i}(n_i \frac{dt_i}{t_i})=n_i\cdot [k(x_i):k]\]
and hence 
\[\sum_{x\in X}\Res_x(\alpha)=\sum_i n_i[k(x_i):k]=0.\]
Thus by the exact sequence \eqref{12}
there exists a form $\omega=g \frac{d f}{f}\in \Omega^1_{k(X)/k}$, $g\in k(X)\setminus\{0\}$, with
\eq{14}{\omega-n_i \frac{d t_i}{t_i}\in \Omega^1_{X,x_i/k},\, i=1,\ldots, r,\text{ and }
 \omega\in \bigcap_{x\in X\setminus\{x_1,\ldots, x_r\} }\Omega^1_{X,x/k}.}
Let $x\in X$ be a closed point. Let $t\in\sO_{X,x}$ be a local parameter at $x$. 
Since $k(x)/k$ is separable 
we can identify the completion of $\sO_{X,x}$ as a $k$-algebra with the formal power series ring
$k(x)[[t]]$ and the completion of $\Omega^1_{X,x/k}$ as a $k$-module with $k(x)[[t]] dt$.
We also have a natural injection 
\[\Omega^1_{k(X)/k}=k(X)\otimes_{\sO_{X,x}}\Omega^1_{X,x/k}\inj 
k(X)\otimes_{\sO_{X,x}}\widehat{\Omega^1_{X,x/k}}= k(x)((t))dt.\]
We denote by $\pi: X\to \Spec k$ the structure map and write $\pi^*\ul{m}$ instead of
$\pi^*m_1\otimes\ldots\otimes \pi^*m_n$. 
Let $\eta\in X$ be the generic point of $X$
which we regard as an element of $\ul{CH}_0(X)(\eta)$, see Remark \ref{rmk:sp-eta}.
We shall compute the local symbol
$(ag\otimes\eta\otimes\pi^*\ul{m}, f)_x$ \eqref{eq:local-symbol}
for all closed points $x \in X$.
We consider five cases.

{\em 1st case: $x=x_i$ for some $i\in\{1,\ldots, r\}$.}
We have by \eqref{11}  $f= t_i^{n_i} u_i$ with $u_i\in \sO_{X,x_i}^\times$.
Thus \eqref{14} (and $(n_i,p)=1$) yields $g\in \sO_{X,x_i}^\times \text{ and } g(x_i)=1.$
Hence
\begin{align}
(ag\otimes\eta\otimes\pi^*\ul{m}, f)_{x_i}
& = v_{x_i}(f)\Tr_{x_i/k} s_{x_i}(ag\otimes \eta\otimes \pi^*\ul{m})\notag
\\
 &= n_i\Tr_{x_i/k}(a\otimes x_i\otimes \pi_i^*\ul{m})\notag
\\
 &=n_i ( a \otimes  x_i\otimes \ul{m}), \notag
\end{align}
where we denote by $\pi_i: x_i\to \Spec k$ the finite morphism induced by $\pi$, $s_{x_i}$ is the specialization map
at $x_i$ \eqref{eq:specializaition} and 
the last equality follows from the projection formula \eqref{eq:proj-formula}
and the fact that the pushforward 
$\pi_{i*}: \ul{CH}_0(X)(x_i)=CH_0(X_{x_i})\to CH_0(X)=\ul{CH}_0(X)(k)$ sends $x_i$ as a cycle on $X_{x_i}$ to $x_i$ viewed as a cycle on $X$.

{\em 2nd case: $x\in |\Div (f)|\setminus\{x_1,\ldots, x_r\}$ and $(v_x(f),p)=1$.}
If $f$ has a pole or a zero at $x\in X\setminus\{x_1,\ldots, x_r\}$ with $(v_x(f),p)=1$, then \eqref{14} implies $g\in\fm_x\subset \sO_{X,x}.$
Thus
\mlnl{(ag\otimes\eta\otimes \pi^*\ul{m}, f)_x= v_{x}(f) \Tr_{x/k}s_x(ag\otimes \eta\otimes \pi^*\ul{m})\\
=v_{x}(f) \Tr_{x/k}(ag(x)\otimes x\otimes s_x(\pi^*\ul{m}))=0.}

{\em 3rd case: $f\in\sO_{X,x}^\times$ and $g\in \sO_{X,x}$.}
In this case we clearly have
\[(ag\otimes\eta\otimes \pi^*\ul{m}, f)_x=0.\]

{\em 4th case: $f\in \sO_{X,x}^\times$ and $g$ has a pole at $x$.}
In $\widehat{\sO_{X,x}}=k(x)[[t]]$ we write 
\[g=\frac{u}{t^r},\quad \text{and}\quad f=a_0(\prod_{i=1}^r(1+a_i t^i))u_r,\] 
where $t\in \sO_{X,x}$ is a local parameter, $u\in \sO_{X,x}^\times$, $u_r\in 1+t^{r+1}k(x)[[t]]$,
$a_i\in k(x)$ and $r\ge 1$. 
For $i\in \{1,\ldots, r\}$ we claim that $a_i=0$ 
if $(i, p)=1$.
Else let $i_0$ be the smallest integer $i\in\{1,\ldots,r\}$  with $(i,p)=1$ and $a_i\neq 0$.
Then
\[\frac{df}{f}= \sum_{i=i_0}^r \frac{i a_i t^{i-1}}{1+a_it^i}dt+\frac{du_r}{u_r}.\]
Since $g\frac{df}{f}$ and $g\frac{du_r}{u_r}$ are regular we obtain 
\[\sum_{i=i_0}^r \frac{ia_iu}{1+a_it^i}\frac{t^{i-1}}{t^r}\in k(x)[[t]],\]
which is only possible if $i_0=0$; a contradiction. 
Since $k(x)$ is perfect we can therefore write $f$ in the form
\[  f=v^p w,\quad \text{with }v\in \sO_{X,x}^\times,\, w\in U_{X,x}^{(r+1)}.\]
Since
$ag\otimes\eta\otimes \pi^*\ul{m}\in{\rm Fil}^{r+1}_x T(\G_a,\ul{CH}_0(X),\ul{\sM})(\eta)$ by 
\eqref{eq:conti},
we get
\[(ag\otimes\eta\otimes \pi^*\ul{m}, f)_x= p(ag\otimes\eta\otimes \pi^*\ul{m}, v)_x
                                               + (ag\otimes\eta\otimes \pi^*\ul{m}, w)_x=0.\]

{\em 5th case: $x\in |\Div (f)|\setminus\{x_1,\ldots, x_r\}$ and $(v_x(f), p) \not=1$.}
In this case we can write $f=t^{v_x(f)}u$ with $t\in \sO_{X,x}$ a local parameter at $x$ and $u\in \sO_{X,x}^\times$ a unit. Then
\[(ag\otimes\eta\otimes \pi^*\ul{m}, f)_x= v_x(f)(ag\otimes\eta\otimes \pi^*\ul{m}, t)_x+
                                                           (ag\otimes\eta\otimes \pi^*\ul{m}, u)_x=0,\]
where the second summand on the right is zero by case 3 or 4.

All together the reciprocity law 
\eqref{eq:reciprocity}
in $T(\G_a,\ul{CH}_0(X),\ul{\sM})(k)$ yields
\begin{align}
0 &=\sum_{x\in X} (ag\otimes\eta\otimes\pi^*\ul{m}, f)_x\notag\\
  &= \sum_{i=1}^5 \sum_{x : \text{case i}}(ag\otimes\eta\otimes\pi^*\ul{m}, f)_x \notag\\
 & =  \sum_{i=1}^r a\otimes n_i x_i\otimes\ul{m}= a\otimes\zeta\otimes\ul{m}.\notag
\end{align}
This finishes the proof.
\qed

\begin{rmk}
\begin{enumerate}
\item
One could generalize the above result to the non-perfect field case, if one could positively answer the following:

{\em Question:}
Let $L \in \sE_k$, $X$ a regular irreducible projective curve over $L$ with
generic point $\eta\in X$, and $x\in X$ a closed point. 
Let $\sM$ be a reciprocity functor.
Is there  a well-defined morphism
\[\rho_{x/L}:\Omega^1_{\eta/L}\otimes_{\Z}\sM(X) \to
   T(\G_a, \sM)(\Spec L),
\]
which satisfies 
for all $a,a'\in \G_a(\eta), b\in \G_m(\eta)$ and $m\in \sM(X)$
\[ \rho_{x/L}((a\frac{db}{b}  + da')\otimes m)
   =(a\otimes m, b)_x \quad ?
\]
\item Actually the above question is only a particular instance of the question, whether for a reciprocity functor $\sM$
the symbol
\[(-,-)_x: \sM(\eta)\times\G_m(\eta)\to \sM(L)\] 
factors through a map $T(\sM,\G_m)(\eta) \to \sM(L)$. 
\end{enumerate}
\end{rmk}

\section{Completion of the proof}
\label{sect:highdim}

We prove the following result,
which implies Theorem \ref{thm:main} as a special case.

\begin{thm}\label{thm:general}
Let $k$ be a perfect field and
let $\sM_1,\ldots, \sM_n$ be reciprocity functors.
Let $X$ be a proper $k$-scheme and assume 
one of the conditions in 
Theorem \ref{thm:huber-kahn}
(so that $\ul{CH}_0(X)^0$ is well-defined).
Then for any {\em perfect} $L \in \sE_k$, we have
\begin{equation}\label{eq:vanish2}
T(\G_a, \ul{CH}_0(X)^0,\sM_1,\ldots, \sM_n)(L)=0.
\end{equation}
\end{thm}

\subsection{Case of curves}\label{sect:curves2}
In this subsection,
we prove Theorem \ref{thm:general} assuming $\dim X=1$.
For a separated $k$-scheme $X$ of finite type,
we write $Z_0(X)^0$ 
for the preimage of $CH_0(X)^0$ in 
the group of zero-cycles $Z_0(X)$ on $X$.

\begin{lem}\label{lem:push-surj}
Let $X$ be a reduced $1$-dimensional proper $k$-scheme
and let $\phi : \tilde{X} \to X$ be its normalization.
Let $Z \subset X$ be the set of non-regular points on $X$.
Assume that for any $z \in Z$ and $z' \in \phi^{-1}(z)$,
we have $k(z')=k(z)=k$.
Then the restriction
\[ 
\phi_* : Z_0(\tilde{X})^0 \to Z_0(X)^0.
\]
of the push-forward 
$\phi_* : Z_0(\tilde{X}) \to Z_0(X)$
is surjective.
\end{lem}
\begin{proof}
We may assume $X$ is connected.
If $X$ is regular, the statement is trivial.
Hence we suppose $Z \not= \emptyset$ and take $z \in Z$
so that $k(z)=k$ by assumption.
Then $Z_0(X)^0$ is generated by elements of the form
$x-mz$ for closed points $x\in X$ and $m:=[k(x):k]$. 
Since $X$ is connected there exist irreducible components 
$X_1,X_2,\ldots X_{n}$ of $X$
such that $x_0:=x\in X_1$, $x_n:=z\in X_n$ and $X_i\cap X_{i+1}\neq\emptyset$, for $1\le i \le  n-1$. Pick $x_i\in X_i\cap X_{i+1} \subset Z$. 
Then $x-mz
= (x-mx_1) + m \sum_{i=1}^{n-1}(x_i-x_{i+1})$ in $Z_0(X)^0$. 
Since the restriction of $\phi$ defines
an isomorphism $\tilde{X} \setminus \phi^{-1}(Z) \to X \setminus Z$,
and since $\phi$ is surjective, 
$(x-mx_1)$ and $(x_i-x_{i+1})$ are in the image of 
$\phi_{*|Z_0(\tilde{X})^0}$ for all $i$, hence so is $x-mz$.
\end{proof}

Now we prove Theorem \ref{thm:general} assuming $\dim X=1$.
Clearly we may assume that $X$ is reduced. Further, by Remark \ref{rem:reduction-to-k} and 
Lemmas \ref{lem:surj}, \ref{rem:base-change}, we may suppose 
that the assumption of Lemma \ref{lem:push-surj}
is satisfied and that $L=k$.
Then Lemma \ref{lem:push-surj} implies that
$\phi$ induces a surjection 
\[
T(\G_a,\ul{CH}_0(\tilde{X})^0,\ul{\sM})(k)
\to
T(\G_a,\ul{CH}_0(X)^0,\ul{\sM})(k),
\]
but $T(\G_a,\ul{CH}_0(\tilde{X})^0,\ul{\sM})(k)=0$
by the result of \S \ref{sect:curves}.
This completes the proof.
\qed

\subsection{General case}
We complete the proof of Theorem \ref{thm:general}. 

\begin{lem}\label{cor:bertini}
Suppose $k$ is an infinite field.
Let $X$ be a proper connected scheme over $k$
such that $\dim X \geq 2$.
Let $\alpha \in CH_0(X)^0$.
Then there exists
a closed $1$-dimensional connected subscheme
$C \subset X$ 
such that
$\alpha$ belongs to the image of
the push-forward
$CH_0(C)^0 \to CH_0(X)^0$.
\end{lem}
\begin{proof}
Given a zero-cycle $\alpha = \sum_{i=1}^r n_i x_i$ 
of degree zero on $X$,
apply the following lemma
with $E=\{ x_1, \dots, x_r \}$ to obtain a zero-cycle $\beta$ on $C$ mapping to $\alpha$.
Since $k_{X,{\rm red}}\to k_{C,{\rm red}}$ is a finite field extension - say of degree $d$, the pushforward
$CH_0(\Spec k_C)\to CH_0(\Spec k_X)$ identifies with $\Z\xr{\cdot d}\Z$. Hence $\beta$ has degree zero.
\end{proof}

\begin{lem}\label{lem:bertini}
Suppose $k$ is an infinite field.
Let $X$ be a proper connected scheme over $k$
such that $\dim X \geq 2$
and let $E$ be a finite set of closed points on $X$. 
Then there exists 
a closed $1$-dimensional connected subscheme
$C \subset X$ which contains $E$.
\end{lem}

This lemma is taken from \cite[13.5.6]{SS}.
For the reader's convenience, 
we include a brief account here.
We take a finite covering $X = \cup_{i \in I} U_i$
by affine open subschemes $U_i$ of $X$.
By enlarging $E$ if necessary,
we may assume the condition
\begin{equation}\label{eq:cond-ak}
\text{if $i, j \in I$ and $U_i \cap U_j \not= \emptyset$,
then there is $x \in E$ such that $x \in U_i \cap U_j$.}
\end{equation}
By \cite[Thm. 1]{AK} 
there is an irreducible closed $1$-dimensional subscheme $C_i \subset U_i$
containing $E\cap U_i$ for each $i \in I$.
Let $C_i'$ be the closure of $C_i$ in $X$
and let $C = \cup_{i \in I}^r C_i'$.
Then $C$ satisfies the required conditions by \eqref{eq:cond-ak}.
\qed

Now we prove Theorem \ref{thm:general} 
assuming $k$ is an infinite field.
We may assume $X$ is connected.
By \S \ref{sect:k-rf} (ii) and Remark \ref{rem:reduction-to-k},
it suffices to show that for any elements 
$a\in k\setminus\{0\}$, $m_i\in \sM_i(k)$ 
and any zero-cycle $\zeta=\sum_{i=1}^r n_i x_i\in Z_0(X)$ 
with $\sum_i n_i [k(x_i): k]=0$ we have
\begin{equation}\label{eq:to-be-killed2}
a\otimes\zeta\otimes m_1\otimes\ldots\otimes m_n=0 \quad \text{in }
T(\G_a,\ul{CH}_0(X)^0,\ul{\sM})(k). 
\end{equation}
By Lemma \ref{cor:bertini},
there is a closed $1$-dimensional subscheme $i : C \hookrightarrow X$
and a zero-cycle of degree zero $\zeta' \in CH_0(C)^0$ 
such that
$\zeta = i_{*}(\zeta')$.
The result of \S \ref{sect:curves2} implies that
\begin{equation}\label{eq:have-been-killed2}
a\otimes\zeta'\otimes m_1\otimes\ldots\otimes m_n=0 \quad \text{in }
T(\G_a,\ul{CH}_0(C)^0,\ul{\sM})(k).
\end{equation}
Now the left hand side of \eqref{eq:to-be-killed2}
is the image of \eqref{eq:have-been-killed2}
by the map induced by $i$, hence is trivial.

It remains to prove Theorem \ref{thm:general} when $k$ is a finite field.
Take a prime $\ell \not= p$
and a $\Z_{\ell}$-extension $k'/k$.
We have proven $\varinjlim_{L/k}T(\G_a,\ul{CH}_0(C)^0,\ul{\sM})(L)=0$,
where the limit is over all finite subextensions of $k'/k$.
(Indeed by the case of an infinite field an element  $\alpha$ in $T(\G_a,\ul{CH}_0(C)^0,\ul{\sM})(L)$ becomes
zero when pulled back to $T(\G_{a,k'},\ul{CH}_0(C_{k'})^0,\ul{\sM}_{k'})(k')$, 
notation as in Remark \ref{rem:reduction-to-k}.
But the corresponding relations involving $\alpha$ are already defined over a finite subextension of $k'/k$.)
By the usual norm argument we conclude that
$T(\G_a,\ul{CH}_0(C)^0,\ul{\sM})(k)$ is $\ell$-primary torsion,
but this group is annihilated by $p$ as well,
hence is trivial.
This completes the proof.
\qed

\subsection{Proof of Corollary  \ref{cor-to-thm:main}}
As before,
it suffices to show that for any elements 
$a\in k\setminus\{0\}$, $m_i\in \sM_i(k)$ 
and $\zeta \in A(k)$
\begin{equation}\label{eq:to-be-killed3}
a\otimes\zeta\otimes m_1\otimes\ldots\otimes m_n=0 \quad \text{in }
T(\G_a,A,\ul{\sM})(k). 
\end{equation}
Since $\zeta \in A(k)$ is a closed point of $A$,
we can regard $[\zeta - 0]$ as a zero cycle of degree zero on $A$,
i.e. $[\zeta - 0] \in \ul{CH}_0(A)^0(k)$.
By Theorem \ref{thm:main},
we get
\begin{equation}\label{eq:have-been-killed1}
a\otimes[\zeta-0]\otimes m_1\otimes\ldots\otimes m_n=0 \quad \text{in }
T(\G_a,\ul{CH}_0(A)^0,\ul{\sM})(k).
\end{equation}
Now the left hand side of \eqref{eq:to-be-killed3} 
is the image of \eqref{eq:have-been-killed1}
by the map induced by the Albanese map $\ul{CH}_0(A)^0 \to A$,
hence is trivial.
\qed

\end{document}